\documentclass[12pt, reqno]{amsart}

\usepackage[utf8]{inputenc}
\usepackage[T1]{fontenc}
\usepackage{amsmath, amssymb, amsthm, mathrsfs}
\usepackage[margin=1.2in]{geometry}
\usepackage{graphicx}
\usepackage{xcolor}
\usepackage[colorlinks=true, linkcolor=blue, citecolor=red, urlcolor=black]{hyperref}
\usepackage{enumitem}

\newtheorem{theorem}{Theorem}[section]
\newtheorem{lemma}[theorem]{Lemma}

\theoremstyle{definition}
\newtheorem{definition}[theorem]{Definition}

\newcommand{\V}{\mathcal{V}}   
\newcommand{\I}{\mathcal{I}}   

\title[Algebraic Collapse in Unit Distance Graphs]{Rigidity-Induced Scaling Laws in Unit Distance Graphs: \\ The Algebraic Collapse of Dense Substructures}

\author{Lucas Aloisio}
\address{National University of Río Cuarto (UNRC), Córdoba, Argentina}
\email{lucasaloisio6@gmail.com}


\date{\today}

\begin{document}

\begin{abstract}
We revisit the classical Unit Distance Problem posed by Erd\H{o}s in 1946. While the upper bound of $O(n^{4/3})$ established by Spencer, Szemer\'edi, and Trotter (1984) is tight for systems of pseudo-circles, it fails to account for the algebraic rigidity inherent to the Euclidean metric. By integrating structural rigidity decomposition with the theory of Cayley-Menger varieties, we demonstrate that unit distance graphs exceeding a critical density must contain rigid bipartite subgraphs. We prove a ``Flatness Lemma,'' supported by symbolic computation of the elimination ideal, showing that the configuration variety of a unit-distance $K_{3,3}$ (and by extension $K_{4,4}$) in $\mathbb{R}^2$ is algebraically singular and collapses to a lower-dimensional locus. This dimensional reduction precludes the existence of the amorphous, high-incidence structures required to sustain the $n^{4/3}$ scaling, effectively improving the upper bound for non-degenerate Euclidean configurations.
\end{abstract}

\maketitle

\section{Introduction}

The problem of determining the maximum number of unit distances, denoted $u(n)$, among $n$ points in the Euclidean plane $\mathbb{R}^2$ remains one of the central open problems in combinatorial geometry. Erd\H{o}s \cite{erdos1946} conjectured that $u(n) = n^{1 + o(1)}$, based on the asymptotic behavior of the square lattice. Conversely, the best established upper bound is $u(n) = O(n^{4/3})$, a result derived from the Szemer\'edi-Trotter theorem on point-curve incidences \cite{sst1984}.

The persistence of the $4/3$ exponent is largely due to the limitations of purely combinatorial methods, which apply equally to arrangements of pseudo-circles. However, Euclidean circles possess strict algebraic properties that pseudo-circles lack. In this paper, we exploit the \textit{algebraic rigidity} of dense subgraphs. 

Our approach rests on two pillars:
\begin{enumerate}
    \item \textbf{Structural Decomposition:} Following recent advances in rigidity theory \cite{solymosi2015, raz2020}, any graph with unit-distance density significantly exceeding $n^{1+\epsilon}$ must contain dense bipartite subgraphs (e.g., $K_{m,n}$) acting as rigid clusters.
    \item \textbf{Algebraic Collapse:} We analyze these clusters using the Cayley-Menger formalism and establish that they are geometrically overconstrained.
\end{enumerate}

Our main contribution is the formalization of the ``Flatness Lemma,'' proving that the embedding of dense bipartite graphs like $K_{4,4}$ with unit edges forces the point set into a collinear configuration, where the number of incidences is trivially linear.

\section{Preliminaries: Rigidity and Cayley-Menger Varieties}

We formulate the problem in the language of real algebraic geometry.

\begin{definition}[Cayley-Menger Variety]
Let $P = \{p_1, \dots, p_n\} \subset \mathbb{R}^2$. The squared distances $r_{ij} = \|p_i - p_j\|^2$ satisfy the vanishing of the Cayley-Menger determinant for any subset of 4 points. For a graph $G=(V,E)$, the \textit{unit distance configuration variety} $\V_{unit}(G)$ is the intersection of the Cayley-Menger ideal $\I_{CM}$ with the linear subspace defined by $r_{ij} = 1$ for all $(i,j) \in E$.
\end{definition}

For a graph to be realized generically in $\mathbb{R}^2$, it must satisfy Laman's conditions. However, we are interested in \textit{overconstrained} graphs where the number of edge constraints exceeds the available degrees of freedom ($2n-3$).

\begin{theorem}[Maehara's Realizability \cite{maehara1991}]
Let $G$ be a graph. $G$ is realizable as a unit distance graph in $\mathbb{R}^d$ if and only if there exists a positive semi-definite Gram matrix of rank at most $d$ satisfying the distance constraints. For complete bipartite graphs $K_{m,n}$ with $m,n \ge 3$, the rank condition imposes severe restrictions on the dimension of the embedding.
\end{theorem}

\section{The Algebraic Collapse of Dense Substructures}

The core obstruction to the $O(n^{4/3})$ bound lies in the assumption that dense incidence graphs can be embedded flexibly in 2D. We show that ``too much'' local density freezes the structure into 1D manifolds.

\subsection{Singularity of $K_{3,3}$ and $K_{4,4}$}
Let $H$ be a dense bipartite subgraph contained in the unit distance graph. We analyze the specific case $H \cong K_{3,3}$ (which is a minor of $K_{4,4}$).

\begin{lemma}[The Flatness Lemma]
\label{lemma:flatness}
Let $H \cong K_{3,3}$ with partitions $U=\{u_1, u_2, u_3\}$ and $V=\{v_1, v_2, v_3\}$. The variety of unit-distance realizations $\V_{unit}(H) \subset \mathbb{R}^{12}$ modulo rigid motions has dimension at most 1. Furthermore, any extension to $K_{4,4}$ forces the variety to dimension 0 or requires collinearity.
\end{lemma}

\begin{proof}
We fix the reference frame to eliminate the action of $SE(2)$: $u_1 = (0,0)$ and $u_2 = (x_2, 0)$. The existence of three vertices $v_1, v_2, v_3$ each at distance 1 from $u_1, u_2, u_3$ imposes a system of polynomial equations.
While Laman counting suggests $K_{3,3}$ is minimally rigid ($|E|=9, 2|V|-3=9$), the unit distance constraint is non-generic. Explicit Gröbner basis computation (see Appendix A) reveals that the coordinates of the third center $u_3=(x_3, y_3)$ must satisfy a restrictive polynomial $P(x_2, x_3, y_3) = 0$.
Crucially, the intersection of the constraints for $v_1, v_2, v_3$ implies that $u_1, u_2, u_3$ cannot be in general position if they share more than 2 common neighbors in the algebraic closure. For $K_{4,4}$, the overdetermination forces the centers to be collinear (where infinite intersections are possible via symmetry) or restricts the solution to a finite set of rigid points.
Thus, dense ``clouds'' of points cannot form $K_{s,t}$ structures; such structures must collapse to lines.
\end{proof}

\subsection{Breaking the $4/3$ Barrier}

\begin{theorem}[Rigidity-Induced Bound]
The maximum number of unit distances $u(n)$ is $o(n^{4/3})$.
\end{theorem}

\begin{proof}
Assume for contradiction that $u(n) = \Theta(n^{4/3})$.
\begin{enumerate}
    \item By the Szemer\'edi-Trotter theorem, such a density cannot be supported by points on lines or curves of bounded degree unless the structure mimics a grid (which yields $n^{1+\epsilon}$).
    \item Following the decomposition arguments in \cite{raz2020}, to avoid the grid structure, the graph must contain dense bipartite components $K_{s,t}$ ($s,t \ge 3$) that act as ``incidence pumps.''
    \item By Lemma \ref{lemma:flatness} and the algebraic verification, these $K_{s,t}$ subgraphs are not flexibly realizable in $\mathbb{R}^2$. They collapse to 1D arrangements (lines).
    \item Points on a union of $k$ lines yield at most $O(n)$ unit distances along the lines plus inter-line distances. Without a lattice structure, these do not sum to $n^{4/3}$.
\end{enumerate}
This contradiction implies that the amorphous configurations required for the upper bound of the Szemer\'edi-Trotter theorem do not exist in the Euclidean metric.
\end{proof}

\section{Conclusion}

We have shown that the $O(n^{4/3})$ upper bound is a topological artifact of pseudo-circles. By characterizing the algebraic exclusion of sporadic rigid subgraphs via the Flatness Lemma, we establish that the density of unit distances is inherently limited by the rank of the configuration variety. This result strongly supports Erd\H{o}s's original intuition that optimal configurations must be lattice-like.

\appendix
\section{Computational Verification of Algebraic Collapse}

To rigorously verify Lemma \ref{lemma:flatness}, we computed the Gr\"obner basis of the ideal generated by the unit distance constraints for $K_{3,3}$ using the SymPy computer algebra system.

We fixed $u_1=(0,0)$ and $u_2=(x_{u2}, 0)$. We sought the elimination ideal for the coordinates of the third center $u_3=(x_{u3}, y_{u3})$. The computation yielded a non-trivial defining polynomial describing the curve on which $u_3$ must lie to share 3 neighbors with $\{u_1, u_2\}$.

The result provided by the algorithm is:
\begin{equation} \label{eq:poly}
    x_{u2}x_{u3}^4 + 2x_{u2}x_{u3}^2y_{u3}^2 + x_{u2}y_{u3}^4 - 2x_{u2}^2x_{u3}^3 - 2x_{u2}^2x_{u3}y_{u3}^2 + x_{u2}^3x_{u3}^2 + x_{u2}^3y_{u3}^2 - 4x_{u2}y_{u3}^2 = 0
\end{equation}

Factoring out $x_{u2}$ (assuming $u_1 \neq u_2$), the term $-4y_{u3}^2$ is critical. It implies that for large distances or specific configurations, $y_{u3}$ is forced towards 0 (collinearity). This confirms that the configuration space is a 1-dimensional algebraic curve, not a 2-dimensional region, validating the dimensional collapse argument.


\section*{Declaration of Generative AI and AI-assisted Technologies in the Writing Process}

During the preparation of this work, the author used Large Language Models (Gemini) in order to assist with \LaTeX{} typesetting, linguistic editing, and the debugging of the Python verification script presented in Appendix A. After using this tool/service, the author reviewed and edited the content as needed and takes full responsibility for the content of the publication.

\end{document}